\documentclass[11pt]{article}
\usepackage{longtable}
\usepackage{amssymb}
\usepackage{amsmath}
\usepackage{amsthm}
\usepackage{enumerate}
\usepackage{tikz}
\usetikzlibrary{positioning, shapes.geometric, calc, intersections}
\tikzstyle{vertex}=[circle,fill=black,inner sep=1pt]
\tikzstyle{vertrect}=[draw,rectangle,inner sep=1pt]
\tikzstyle{vertdia}=[draw,diamond,inner sep=1pt]

\theoremstyle{plain}
      \newtheorem{theorem}{Theorem}
      \newtheorem{lemma}[theorem]{Lemma}
            
      \newtheorem{corollary}[theorem]{Corollary}

      \newtheorem{conjecture}{Conjecture}

            \theoremstyle{definition}
      \newtheorem{definition}[theorem]{Definition}

      \theoremstyle{remark}

\title{Multicolor Sunflowers}

\author{
{\Large{Dhruv Mubayi}}\thanks{
\footnotesize {Department of Mathematics, Statistics, and Computer Science, University of Illinois, Chicago, IL 60607, \texttt{Email:mubayi@uic.edu}}. Research partially supported by NSF grants DMS-0969092 and DMS-1300138}
\and
{\Large{Lujia Wang}}\thanks{
\footnotesize {Department of Mathematics, Statistics, and Computer Science, University of Illinois, Chicago, IL 60607, \texttt{Email:lwang203@uic.edu}}.}}

\date{\today}

\begin{document}
\maketitle
\begin{abstract}
A sunflower is a collection of distinct sets such that the intersection of any two of them is the same as the common intersection $C$ of all of them, and $|C|$ is smaller than each of the sets. A longstanding conjecture due to Erd\H{o}s and Szemer\'{e}di states that the maximum size of a family of subsets of $[n]$ that contains no sunflower of fixed size $k>2$ is exponentially smaller than $2^n$ as $n\rightarrow\infty$. We consider this problem for multiple families. In particular, we obtain sharp or almost sharp bounds on the sum and product of $k$ families  of subsets of $[n]$ that together contain no sunflower of size $k$ with one set from each family. For the sum, we prove that the maximum is 
$$(k-1)2^n+1+\sum_{s=n-k+2}^{n}\binom{n}{s}$$
for all $n \ge k \ge 3$, and for the $k=3$ case of the product, we prove that it is between 
$$\left(\frac{1}{8}+o(1)\right)2^{3n}\qquad \hbox{ and } \qquad (0.13075+o(1))2^{3n}.$$
\end{abstract} 

\section{Introduction}
 Throughout the paper, we write $[n]=\{1,\ldots,n\}$, $2^{[n]}=\{S:S\subset[n]\}$ and $\binom{[n]}{s}=\{S:S\subset[n], |S|=s\}$. A \emph{sunflower} (or \emph{strong $\Delta$-system}) with $k$ petals is a collection of $k$ sets ${\cal S}=\{S_1,\ldots,S_k\}$ such that $S_i\cap S_j=C$ for all $i\neq j$, and $S_i\setminus C\neq\emptyset$ for all $i\in[k]$. The common intersection $C$ is called the \emph{core} of the sunflower and the sets $S_i\setminus C$ are called the \emph{petals}. In 1960, Erd\H{o}s and Rado~\cite{ER60} proved a fundamental result regarding the existence of sunflowers in a large family of sets of uniform size, which is now referred to as the \emph{sunflower lemma}. It states that if $\cal A$ is a family of sets of size $s$ with $|{\cal A}|>s!(k-1)^s$, then $\cal A$ contains a sunflower with $k$ petals. Later in 1978, Erd\H{o}s and Szemer\'{e}di~\cite{ES78} gave the following upper bound when the underlying set has $n$ elements.
\begin{theorem}[Erd\H{o}s, Szemer\'{e}di~\cite{ES78}]
There exists a constant $c$ such that if ${\cal A}\subset 2^{[n]}$ with $|{\cal A}|>2^{n-c\sqrt{n}}$ then $\cal A$ contains a sunflower with $3$ petals.\label{thmES}
\end{theorem}
In the same paper, they conjectured  that for $n$ sufficiently large, the maximum number of sets in a family ${\cal A}\subset 2^{[n]}$ with no sunflowers with three petals is at most $(2-\epsilon)^n$ for some absolute constant $\epsilon>0$. This conjecture remains open, and is closely related to the algorithmic problem of matrix multiplication, see~\cite{ASU13}.

Similar problems have been studied for systems of sets where only the size (rather than the actual set) of pairwise intersections is fixed. A \emph{weak $\Delta$-system} of size $k$ is a collection of $k$ sets ${\cal S}=\{S_1,\ldots,S_k\}$ such that $|S_i\cap S_j|=|S_1\cap S_2|$ for all $i\neq j$. Thus, a sunflower is a weak $\Delta$-system but not vice versa. In 1973, Deza~\cite{D73} gave the criterion for a weak $\Delta$-system to be a sunflower: If ${\cal F}$ is an $s$-uniform weak $\Delta$-system with $|{\cal F}|>s^2-s+1$, then $\cal F$ is a sunflower. The lower bound can be achieved only if the projective plane $PG(2,s)$ exists. This was shown by van Lint~\cite{L73} later in the same year. Erd\H os posed the problem of determining the largest size of a family ${\cal A} \subset 2^{[n]}$ that contains no weak $\Delta$-system of a fixed size.
The problem was solved by Frankl and R\"{o}dl~\cite{FR87} in 1987. They proved that given $k\ge 3$, there exists a constant $\epsilon=\epsilon(k)$ so that for every ${\cal A}\subset 2^{[n]}$ with $|{\cal A}|>(2-\epsilon)^n$, $\cal A$ contains a weak $\Delta$-system of size $k$. 

A natural way to generalize problems in extremal set theory is to  consider versions for multiple families or so-called multicolor or cross-intersecting problems. Beginning with the famous Erd\H{o}s-Ko-Rado theorem~\cite{EKR61}, which states that an intersecting family of $k$-element subsets of $[n]$ has size at most $\binom{n-1}{k-1}$, provided $n\geq 2k$, several generalizations were proved for multiple families that are cross-intersecting. In particular, Hilton~\cite{H77} showed in 1977 that if $t$ families ${\cal A}_1,\ldots,{\cal A}_t\subset\binom{[n]}{k}$ are cross intersecting
(meaning that $A_i \cap A_j \neq\emptyset$ for all $(A_i, A_j) \in {\cal A}_i\times {\cal A}_j$) and  if $n/k\leq t$, then $\sum_{i=1}^t|{\cal A}_i|\leq t\binom{n-1}{k-1}$. On the other hand, results of Pyber~\cite{P86} in 1986, that were later slightly refined by Matsumoto and Tokushige~\cite{MT89} and Bey~\cite{B05}, showed that if two families ${\cal A}\subset\binom{[n]}{k}$, ${\cal B}\subset\binom{[n]}{l}$ are cross-intersecting and $n\geq\max\{2k,2l\}$, then $|{\cal A}||{\cal B}|\le\binom{n-1}{k-1}\binom{n-1}{l-1}$. These are the first results about bounds on sums and products of the size of cross-intersecting families. More general problems were considered recently, for example for cross $t$-intersecting families (i.e. pair of sets from distinct families have intersection of size at least $t$) and $r$-cross intersecting families (any $r$ sets have a nonempty intersection where each set is picked from a distinct family) and labeled crossing intersecting families, see~\cite{B08,FT11,FLST14}.  A more systematic study of multicolored extremal problems (with respect to the sum of the sizes of the families) was begun by Keevash, Saks, Sudakov, and Verstra\"ete~\cite{KSSV}, and continued in~\cite{BKS, KS}. 
Cross-intersecting versions of Erd\H os' problem on weak $\Delta$-systems mentioned above (for the product of the size of  two families) were  proved by Frankl and R\"odl~\cite{FR87} and by the first author and R\"odl~\cite{MR}.

In this note, we consider  multicolor versions of sunflower theorems. Quite surprisingly, these basic questions appear not to have been studied in the literature.

\begin{definition}
 Given families of sets ${\cal A}_i\subset 2^{[n]}$ for $i=1,\ldots,k$, a \emph{multicolor sunflower} with $k$ petals is a collection of sets $A_i\in {\cal A}_i$, $i=1,\ldots,k$, such that $A_i\cap A_j=C$ for all $i\neq j$, and $A_i\setminus C\neq\emptyset$, for all $i\in[k]$.
 Say that ${\cal A}_1, \ldots, {\cal A}_k$ is sunflower-free if it contains no multicolor sunflower with $k$ petals.
  \end{definition}
  
  For any $k$ families that are sunflower-free, the problem of upper bounding the size of any  single family is uninteresting, since there is no restriction on a particular family. So we are interested in the sum and product of the sizes of these families.

Given integers $n$ and $k$, let
\begin{equation*}
{\cal F}(n,k)=\{\{{\cal A}_i\}_{i=1}^k:{\cal A}_i\subset 2^{[n]} \text{ for  $i\in [k]$  and  ${\cal A}_1,{\cal A}_2,\ldots,{\cal A}_k$ is sunflower-free}\}.
\end{equation*}
We define
$$S(n,k):=\max_{\{{\cal A}_i\}_{i=1}^k\in{\cal F}(n,k)}\sum_{i=1}^k|\mathcal{A}_i|,$$
and
$$P(n,k):=\max_{\{{\cal A}_i\}_{i=1}^k\in{\cal F}(n,k)}\prod_{i=1}^k|\mathcal{A}_i|.$$

\section{Main Results}
Our two main results are sharp or nearly sharp estimates on $S(n,k)$ and $P(n,3)$.  By Theorem~\ref{thmES} we obtain that 
$$S(n,3)\le 2\cdot2^n+2^{n-c\sqrt{n}}.$$
Indeed, if $|{\cal A}|+|{\cal B}|+|{\cal C }|$ is larger than the RHS above then $|{\cal A}\cap{\cal B}\cap{\cal C}|>2^{n-c\sqrt{n}}$ by the pigeonhole principle and we find a sunflower in the intersection which contains a multicolor sunflower. Our first result reduces the term $2^{n-c\sqrt{n}}$ to obtain an exact result.
\begin{theorem}For $n\ge k\ge 3$
$$S(n,k)= (k-1)2^n+1+\sum_{s=n-k+2}^n\binom{n}{s}.$$\label{thm2}
\end{theorem} 
The problem of determining $P(n,k)$ seems more difficult than that of determining $S(n,k)$. Our bounds for general $k$ are quite far apart, but in the case $k=3$ we can refine our argument to obtain a better bound.

\begin{theorem}
$$\left(\frac{1}{8}+o(1)\right)2^{3n}\le P(n,3)\le \left(0.13075+o(1)\right)2^{3n}.$$\label{thm3}
\end{theorem}
We conjecture that the lower bound is tight.

\begin{conjecture} For each fixed $k \ge 3$, 
$$P(n,k)=\left(\frac{1}{8}+o(1)\right)2^{kn}.$$
\end{conjecture}

In the next two subsections we give the proofs of Theorems \ref{thm2} and \ref{thm3}.

\subsection{Sums}
In order to prove Theorem~\ref{thm2}, we first deal with  $s$-uniform families and prove a stronger result. Given a multicolor sunflower ${\cal S}$, define its \emph{core size} to be $c({\cal S})=|C|$.

\begin{lemma}
Given integers $s\ge 1$ and $c$ with $0\le c\le s-1$, let $n$ be an integer such that $n\ge c+k(s-c)$. For $ i=1,\ldots,k$, let $\mathcal{A}_i\subset \binom{[n]}{s}$ such that $\{\mathcal{A}_i\}_{i=1}^k$ contains no multicolor sunflower with $k$ petals and core size $c$. Then
$$\sum_{i=1}^k|\mathcal{A}_i|\le (k-1)\binom{n}{s}.$$
Furthermore, this bound is tight.
\label{thm1}
\end{lemma}
\begin{proof}
Randomly take an ordered partition of $[n]$ into $k+2$ parts $X_1, X_2,\ldots, X_{k+2}$ such that $|X_1|=n-(c+k(s-c)), |X_2|=c$, and $|X_i|=s-c$ for $i=3,\ldots, k+2$, with uniform probability for each partition. For each partition, construct the bipartite graph 
$$G=(\{\mathcal{A}_i:i=1,\ldots,k\}\cup\{X_2\cup X_j:j\in [3,k+2]\},E)$$ where a pair $\{\mathcal{A}_i,X_2\cup X_j\}\in E$ if and only if $X_2\cup X_j\in\mathcal{A}_i$. If there exists a perfect matching in $G$, then we will get a multicolor sunflower with $k$ petals and core size $c$, since $X_2$ will be the core. This shows that $G$ has matching number at most $k-1$. Then K\"{o}nig's theorem implies that the random variable $|E(G)|$ satisfies
\begin{equation} \label{k} |E(G)|\le (k-1)k.\end{equation}
Another way to count the edges of $G$ is through the following expression:
\begin{equation}
\begin{split}
|E(G)|=\sum_{i=1}^k\sum_{j=3}^{k+2}\chi_{\{X_2\cup X_j\in \mathcal{A}_i\}},
\end{split}\label{ineq1}
\end{equation}
where $\chi_A$ is the characteristic function of the event $A$. Taking expectations and using (\ref{k}) we obtain
$$\mathbb{E}\left(\sum_{i=1}^k\sum_{j=3}^{k+2}\chi_{\{X_2\cup X_j\in \mathcal{A}_i\}}\right) \le (k-1)k.$$
 By linearity of expectation,
\begin{equation*}
\begin{split}
&\mathbb{E}\left(\sum_{i=1}^k\sum_{j=3}^{k+2}\chi_{\{X_2\cup X_j\in \mathcal{A}_i\}}\right)=\sum_{i=1}^k\sum_{j=3}^{k+2}\mathbb{P}\left(X_2\cup X_j\in \mathcal{A}_i\right)=\sum_{i=1}^k\sum_{j=3}^{k+2}\sum_{A\in\mathcal{A}_i}\mathbb{P}\left(A=X_2\cup X_j\right).
\end{split}
\end{equation*}
The probability that a set $A$ is partitioned as $X_2\cup X_j$ is the same as the probability that $A$ is partitioned into two ordered parts of sizes $c$ and $s-c$, and $[n]\setminus A$ has an ordered partitioned into $k$ parts with one of the parts of size $n-(c+k(s-c))$ and $k-1$ of them of size $s-c$. Hence for any $A\in {\cal A}_i$,
\begin{equation*}
\begin{split}
\mathbb{P}{\left(A=X_2\cup X_j\right)}
&=\frac{\binom{|A|}{c}\binom{n-|A|}{n-(c+k(s-c))}\prod_{i=1}^{k-1}\binom{(k-i)(s-c)}{s-c}}{\binom{n}{c+k(s-c)}\binom{c+k(s-c)}{c}\prod_{i=0}^{k-1}\binom{(k-i)(s-c)}{s-c}}\\
&=\frac{\binom{s}{c}\binom{n-s}{n-(c+k(s-c))}\prod_{i=1}^{k-1}\binom{(k-i)(s-c)}{s-c}}{\binom{n}{c+k(s-c)}\binom{c+k(s-c)}{c}\binom{k(s-c)}{s-c}\prod_{i=1}^{k-1}\binom{(k-i)(s-c)}{s-c}}\\
&=\frac{1}{\binom{n}{s}}.
\end{split}
\end{equation*}
So we have 
$$\mathbb{E}\left(\sum_{i=1}^k\sum_{j=3}^{k+2}\chi_{\{X_2\cup X_j\in \mathcal{A}_i\}}\right)=\sum_{i=1}^k\sum_{j=3}^{k+2}\sum_{A\in\mathcal{A}_i}\frac{1}{\binom{n}{s}}=\sum_{i=1}^k|\mathcal{A}_i|\frac{k}{\binom{n}{s}}.$$
Hence by (\ref{ineq1}),
$$\sum_{i=1}^k|\mathcal{A}_i|\le (k-1)\binom{n}{s}.$$

The bound shown above is tight, since we can take ${\cal A}_1={\cal A}_2=\ldots={\cal A}_{k-1}=\binom{[n]}{s}$, and ${\cal A}_k=\emptyset$.
\end{proof}

 Now we use this lemma to prove Theorem~\ref{thm2}. 
\bigskip

{\bf Proof of Theorem~\ref{thm2}.}  Recall that $n \ge k \ge 3$ and we are to show that
$$S(n,k)= (k-1)2^n+1+\sum_{s=n-k+2}^n\binom{n}{s}.$$
To see the upper bound, given families $\{{\cal A}_i\}_{i=1}^k\in {\cal F}(n,k)$, we define ${\cal A}_{i,s}={\cal A}_i\cap \binom{[n]}{s}$ for each $i\in [k]$ and integer $s\in [0,n]$. This gives a partition of each family ${\cal A}_i$ into $n+1$ subfamilies. Since families $\{{\cal A}_i\}_{i=1}^k$ contain no multicolor sunflowers with $k$ petals, neither do $\{{\cal A}_{i,s}\}_{i=1}^k$ for all $s\in[0,n]$. Now, for each $s=1, 2,\ldots, n-k+1$ let 
$$c=\max\left\{0,\, \frac{n-ks}{1-k}\right\}.$$ Then $0 \le c \le s-1$, and $n\ge c+k(s-c)$. Therefore, by Lemma~\ref{thm1}, for $1 \le s\le n-k+1$,
$$\sum_{i=1}^k|{\cal A}_{i,s}|\le(k-1)\binom{n}{s}.$$
For $s>n-k+1$, notice that a trivial bound for this sum is $k\binom{n}{s}$. So we get,
\begin{equation*}
\begin{split}
\sum_{i=1}^k|{\cal A}_{i}|&=\sum_{i=1}^{k}\sum_{s=0}^{n}|{\cal A}_{i,s}|\\
&=\sum_{s=0}^{n}\sum_{i=1}^k|{\cal A}_{i,s}|\\
&=\sum_{i=1}^k|{\cal A}_{i,0}|+\sum_{s=1}^{n-k+1}\sum_{i=1}^k|{\cal A}_{i,s}|+\sum_{s=n-k+2}^{n}\sum_{i=1}^k|{\cal A}_{i,s}|\\
&\leq k\binom{n}{0}+ \sum_{s=1}^{n-k+1}(k-1)\binom{n}{s}+ \sum_{s=n-k+2}^nk\binom{n}{s}\\
&\leq \sum_{s=0}^n(k-1)\binom{n}{s}+\binom{n}{0}+\sum_{s=n-k+2}^n\binom{n}{s}\\
&=(k-1)2^n+1+\sum_{s=n-k+2}^n\binom{n}{s}.
\end{split}
\end{equation*}
The lower bound is obtained by the following example: ${\cal A}_i=2^{[n]}$ for $i=1\ldots,k-1$ and ${\cal A}_k=\{\emptyset\}\cup\{S\subset[n]: |S|\ge n-k+2\}$. To see that $\{{\cal A}_i\}_{i=1}^{k}$ contains no multicolor sunflower, notice that any multicolor sunflower uses a set from ${\cal A}_k$. The empty set does not lie in any sunflowers. So if a set of size at least $n-k+2$ appeared in a sunflower with $k$ petals, it requires at least $k-1$ other points to form such a sunflower, but then the total number of points in this sunflower is at least $n+1$, a contradiction.
\qed

\subsection{Products}
From the bound on the sum of the families that do not contain a multicolor sunflower, we deduce the following bound on the product by using AM-GM inequality.
\begin{corollary} \label{cor1}Fix $k\ge 3$. As $n\rightarrow\infty$,
$$\left(\frac{1}{8}+o(1)\right)2^{kn}\le P(n,k)\le\left(\left(\frac{k-1}{k}\right)^k+o(1)\right)2^{kn}.$$
\end{corollary}
\begin{proof}
The upper bound follows from Theorem~\ref{thm2} and the AM-GM inequality,
$$\prod_{i=1}^k|\mathcal{A}_i|\le\left(\frac{\sum_{i=1}^k|\mathcal{A}_i|}{k}\right)^k\le\left((1+o(1))\frac{(k-1)2^n}{k}\right)^k=(1+o(1))\left(\frac{k-1}{k}\right)^k2^{kn}.$$

For the lower bound, we take 
$${\cal A}_{1}={\cal A}_2=\{S\subset[n]:1\in S\mbox{ or }|S|\ge n-1\},$$
$${\cal A}_3=\{S\subset[n]:1\notin S\mbox{ or }|S|\ge n-1\},$$
and ${\cal A}_4={\cal A}_5=\ldots={\cal A}_{k}=2^{[n]}$. A multicolor sunflower with $k$ petals must use three sets from ${\cal A}_{1},{\cal A}_2$, and ${\cal A}_{3}$, call them $A_1, A_2, A_3$ respectively. These three sets form a multicolor sunflower with three petals. If any of these sets is of size at least $n-1$, then it will be impossible to form a 3-petal sunflower with the other two sets. So by their definitions, we have $1\in A_1\cap A_2$, but $1\notin A_3$, which implies $A_1\cap A_2\neq A_1\cap A_3$, a contradiction. So the families $\{{\cal A}_i\}_{i=1}^k$ contain no multicolor sunflowers with $k$ petals. The sizes of these families are $|{\cal A}_{1}|=|{\cal A}_{2}|=|{\cal A}_{3}|=(\frac{1}{2}+o(1))2^n$, and $|{\cal A}_{i}|=2^n$ for $i\ge 4$. Thus,
$$\prod_{i=1}^k|\mathcal{A}_i|=\left(\frac{1}{8}+o(1)\right)2^{kn}.$$
\end{proof}

For any positive integer $k$ we have $(\frac{k-1}{k})^k<1/e$, so Corollary~\ref{cor1} implies the upper bound $(1/e+o(1))2^{kn}$ for all $k\ge 3$.  For $k=3$, we will improve the factor in the upper bound from $(2/3)^3=0.29629\cdots$ to approximately $0.131$, which is quite close to our conjectured value of $0.125$.  We need the following lemma. 

\begin{lemma}
Let $G=(V_1\cup V_2,E)$ be a bipartite graph with $|V_1|=|V_2|=3$ and $d(v)\le 2$ for all $v\in V_2$. If the maximum size of a matching in $G$ is at most two, then $G$ is a subgraph of one of the following three graphs.

$\bullet$ $G_1$: a copy of $K_{2,3}$ with the part of size two in $V_1$ and the part of size three in $V_2$

$\bullet$  $G_2$: two vertex disjoint copies of the path with two edges

$\bullet$  $G_3$: a path with four edges whose endpoints are in $V_1$
 
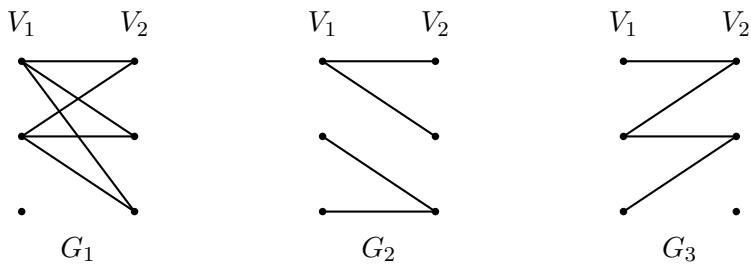
\begin{center}
\begin{tikzpicture}
  \node at (0,0) [vertex] {};
  \node at (0,1) [vertex] {};
  \node at (0,2) [vertex] {};
  \node at (1.5,0) [vertex] {};
  \node at (1.5,1) [vertex] {};
  \node at (1.5,2) [vertex] {};
  \node at (0,2.5)  {$V_1$};
  \node at (1.5,2.5) {$V_2$};
  \node at (0.75,-0.5) {$G_1$};
  \node at (4,0) [vertex] {};
  \node at (4,1) [vertex] {};
  \node at (4,2) [vertex] {};
  \node at (5.5,0) [vertex] {};
  \node at (5.5,1) [vertex] {};
  \node at (5.5,2) [vertex] {};
  \node at (4,2.5)  {$V_1$};
  \node at (5.5,2.5) {$V_2$};
  \node at (4.75,-0.5) {$G_2$};
  \node at (8,0) [vertex] {};
  \node at (8,1) [vertex] {};
  \node at (8,2) [vertex] {};
  \node at (9.5,0) [vertex] {};
  \node at (9.5,1) [vertex] {};
  \node at (9.5,2) [vertex] {};
  \node at (8,2.5)  {$V_1$};
  \node at (9.5,2.5) {$V_2$};
  \node at (8.75,-0.5) {$G_3$};
    \draw[thick]  (0,2) to (1.5,2);
    \draw[thick]  (0,2) to (1.5,1);
    \draw[thick]  (0,2) to (1.5,0);
    \draw[thick]  (0,1) to (1.5,2);
    \draw[thick]  (0,1) to (1.5,1);
    \draw[thick]  (0,1) to (1.5,0);
    \draw[thick]  (4,2) to (5.5,2);
    \draw[thick]  (4,2) to (5.5,1);
    \draw[thick]  (4,1) to (5.5,0);
    \draw[thick]  (4,0) to (5.5,0);
    \draw[thick]  (8,2) to (9.5,2);
    \draw[thick]  (8,1) to (9.5,2);
    \draw[thick]  (8,1) to (9.5,1);
    \draw[thick]  (8,0) to (9.5,1);
  \end{tikzpicture}
\end{center}
\label{strcture}
\end{lemma}
\begin{proof}
By K\"{o}nig's theorem, the minimum vertex cover of $G$ has size at most two. Suppose $G$ has a vertex that covers all the edges.  Then $G$ is a subgraph of either a $K_{1,3}$ whose degree three vertex is in $V_1$ or a $K_{1,2}$ whose degree two vertex is in $V_2$. Both of these are subgraphs of $G_1$.

Now we may assume that a minimum vertex cover $S$ of $G$ has size two. If $S \subset V_1$, then $G$ is a subgraph of $G_1$. Next we assume that $S=\{u,v\}$  with $u\in V_1, v\in V_2$. If $d(u)=3$, then $d(v)\le 2$ implies that $G\subset G_1$. If $d(u)=2$ and $uv\in E(G)$, then $G\subset  G_1$. If $d(u)=2$ and $uv\notin E(G)$, then $G\subset G_2$. If $d(u)=1$, then clearly $G\subset G_3$. The remaining case is that  $S \subset V_2$, and it is obvious that $G\subset G_1$ or $G\subset G_3$.
\end{proof}

  We now have the necessary ingredients to prove Theorem~\ref{thm3}.

{\bf Proof of Theorem~\ref{thm3}.}  Recall that $n \ge k \ge 3$ and we are to show that
$$\left(\frac{1}{8}+o(1)\right)2^{3n}\le P(n,3)\le \left(0.13075+o(1)\right)2^{3n}.$$
The lower bound follows from Corollary~\ref{cor1}. The upper bound is proved using a similar idea as the proof of Lemma~\ref{thm1}, although the graph statistic considered is more complicated. First notice that given families $\mathcal{A}_i\subset 2^{[n]}, i=1, 2, 3$ that are in ${\cal F}(n,3)$, such that $\prod_{i=1}^3|{\cal A}_i|$ is maximized, we may assume that the common part of the three families $\bigcap_{i=1}^3{\cal A}_i=\emptyset$. To see this, let ${\cal A}_c=\bigcap_{i=1}^3{\cal A}_i$. By Theorem~\ref{thmES}, $|{\cal A}_c|\le 2^{n-c\sqrt{n}}=o(2^n)$, otherwise a $3$-petal multicolor sunflower exists. Notice also that from the lower bound, $|{\cal A}_i|=\Theta(2^n)$ for all $i\in[3]$, we have
$$\prod_{i=1}^3|{\cal A}_i|=\prod_{i=1}^3(|{\cal A}_i\setminus{\cal A}_c|+|{\cal A}_c|)=\prod_{i=1}^3(|{\cal A}_i\setminus{\cal A}_c|+o(2^n))=\prod_{i=1}^3|{\cal A}_i\setminus{\cal A}_c|+o(2^{3n}).$$
Hence, it suffices to show a bound on $\prod_{i=1}^3|{\cal A}_i\setminus{\cal A}_c|$.

We uniformly take an ordered partition of $[n]=X_1\cup X_2\cup X_3\cup X_4$ at random such that the parts $X_2, X_3$ and $X_4$ are nonempty. So there are 
$$p(n)=4^n-3\cdot3^n+3\cdot2^n-1=4^n+O(3^n)$$ such partitions, each is chosen with probability $1/p(n)$. Again, we construct a bipartite graph $G=(V_1\cup V_2,E)$, where $V_1=\{\mathcal{A}_i:i=1,2,3\}$ and $V_2=\{X_1\cup X_j: j=2,3,4\}$, and a pair $\{\mathcal{A}_i,X_1\cup X_j\}\in E$ if and only if $X_1\cup X_j\in \mathcal{A}_i$. A perfect matching in $G$ gives rise to a multicolor sunflower with three petals. Hence the maximum size of a matching in  $G$ is at most two. Moreover,  since ${\cal A}_c=\emptyset$, the degrees of vertices in $V_2$ are at most two. We may now apply Lemma~\ref{strcture} to deduce that $G$ is a subgraph of $G_i$ for some $i=1,2,3$. 

 Let $m_2(G)$ be the number of matchings in $G$ of size two and $t(G)$ be the number of five vertex subgraphs of $G$ comprising a degree two vertex $v \in V_2$, the two edges incident to it, and an additional isolated edge. Observe that $m_2(G_1)+t(G_1)=6+0=6,m_2(G_2)+t(G_2)=4+2=6,$ and $m_2(G_3)+t(G_3)=3+2=5$. Since $G\subset G_i$ for some $i=1,2,3$, we have
 \begin{equation} \label{mf} m_2(G)+t(G) \le \max_{i\in[3]}\left(m_2(G_i)+t(G_i)\right)= 6.
 \end{equation}
 Let
 $$P=\sum_{\substack{(\mathcal{B}_1,\mathcal{B}_2)\in V_1^2\\{\cal B}_1\neq{\cal B}_2}}\sum_{\substack{(Y_1,Y_2)\in V_2^2\\Y_1\neq Y_2}}\frac{1}{2}\chi_{\{Y_i\in \mathcal{B}_{i}: i=1,2\}}$$
 and 
$$Q=\sum_{\substack{(\mathcal{B}_1,\mathcal{B}_2,\mathcal{B}_3)\in V_1^3\\{\cal B}_i\neq{\cal B}_j, i\neq j}}\sum_{\substack{(Y_1,Y_2)\in V_2^2\\Y_1\neq Y_2}}\frac{1}{2}\chi_{\{Y_1\in \mathcal{B}_{1},Y_2\in \mathcal{B}_{2},Y_2\in \mathcal{B}_{3}\}}.$$
Then (\ref{mf}) implies that  
   \begin{equation}
\begin{split}
&\mathbb{E}(P+Q)\le 6.
\end{split}\label{ineq2}
\end{equation}
By linearity of expectation, to calculate $\mathbb{E}(P+Q)$, we just need to calculate the expectations of $\chi_{\{Y_i\in \mathcal{B}_{i}: i=1,2\}}$ and $\chi_{\{Y_1\in \mathcal{B}_{1},Y_2\in \mathcal{B}_{2},Y_2\in \mathcal{B}_{3}\}}$.

Call a pair of sets $(B_i\in \mathcal{B}_{i}: i=1,2)$ \emph{good} if $B_1\setminus B_2\neq\emptyset$, $B_2\setminus B_1\neq\emptyset$ and $B_1\cup B_2\neq[n].$ Conversely, a pair is called \emph{bad}, if either $B_1\subset B_2$, or $B_2\subset B_1$, or $B_1\cup B_2=[n]$. We see that bad pairs induce partitions on $[n]$ into at most three parts, which shows that the number of bad pairs is $O(3^n)$. Now, for each good pair $(B_1, B_2)$, there exists a unique partition $$[n]= (B_1\cap B_2) \cup 
(B_1\setminus B_2)\cup (B_2\setminus B_1) \cup ([n]\setminus(B_1 \cup B_2))$$ such that $Y_1=B_1, Y_2=B_2$. Therefore,
\begin{equation*}
\begin{split}
\mathbb{E}\left(\chi_{\{Y_i\in \mathcal{B}_{i}: i=1,2\}}\right)&=\mathbb{P}\left(Y_i\in \mathcal{B}_{i}: i=1,2\right) \\
&=\sum_{B_1\in\mathcal{B}_{1}}\sum_{B_2\in\mathcal{B}_{2}}\mathbb{P}\left(Y_1=B_1,Y_2=B_2\right)\\
&=\sum_{(B_i\in {\cal B}_i: i=1,2)\text{ is good}}\frac{1}{p(n)}\\
&=\frac{\#\{\text{good pairs } (B_i\in {\cal B}_i: i=1,2)\}}{4^n+O(3^n)}\\
&=\frac{|{\cal B}_1||{\cal B}_2|+O(3^n)}{4^n+O(3^n)}\\
&=(1+o(1))\frac{|{\cal B}_1||{\cal B}_2|}{4^n}. 
\end{split}
\end{equation*}

Similarly,
\begin{equation*}
\begin{split}
\mathbb{E}\left(\chi_{\{Y_1\in \mathcal{B}_{1},Y_2\in \mathcal{B}_{2},Y_2\in \mathcal{B}_{3}\}}\right)&=\mathbb{P}\left(Y_1\in \mathcal{B}_{1},Y_2\in \mathcal{B}_{2},Y_2\in \mathcal{B}_{3}\right)\\
&=\sum_{B_1\in\mathcal{B}_{1}}\sum_{B_2\in\mathcal{B}_{2}\cap\mathcal{B}_3}\mathbb{P}\left(Y_1=B_1,Y_2=B_2\right)\\
&=(1+o(1))\frac{|\mathcal{B}_{1}||\mathcal{B}_{2}\cap\mathcal{B}_3|}{4^n}.
\end{split}
\end{equation*}
Consequently, $\mathbb{E}(P+Q)$ is equal to
\begin{equation*}
\begin{split}
&\sum_{\substack{(\mathcal{B}_1,\mathcal{B}_2)\in V_1^2\\{\cal B}_1\neq{\cal B}_2}}\sum_{\substack{(Y_1,Y_2)\in V_2^2\\Y_1\neq Y_2}}\frac{(1+o(1))}{2}\frac{|\mathcal{B}_{1}||\mathcal{B}_{2}|}{4^n}+\sum_{\substack{(\mathcal{B}_1,\mathcal{B}_2,\mathcal{B}_3)\in V_1^3\\{\cal B}_i\neq{\cal B}_j, i\neq j}}\sum_{\substack{(Y_1,Y_2)\in V_2^2\\Y_1\neq Y_2}}\frac{(1+o(1))}{2}\frac{|\mathcal{B}_{1}||\mathcal{B}_{2}\cap\mathcal{B}_3|}{4^n}\\
&=\sum_{\substack{(\mathcal{B}_1,\mathcal{B}_2)\in V_1^2\\{\cal B}_1\neq{\cal B}_2}}\frac{(6+o(1))}{2}\frac{|\mathcal{B}_{1}||\mathcal{B}_{2}|}{4^n}+\sum_{\substack{(\mathcal{B}_1,\mathcal{B}_2,\mathcal{B}_3)\in V_1^3\\{\cal B}_i\neq{\cal B}_j, i\neq j}}\frac{(6+o(1))}{2}\frac{|\mathcal{B}_{1}||\mathcal{B}_{2}\cap\mathcal{B}_3|}{4^n}\\
&=\frac{6+o(1)}{4^n}(|\mathcal{A}_{1}||\mathcal{A}_{2}|+|\mathcal{A}_{2}||\mathcal{A}_{3}|+|\mathcal{A}_{3}||\mathcal{A}_{1}|+|\mathcal{A}_{1}||\mathcal{A}_{2}\cap\mathcal{A}_3|+|\mathcal{A}_{2}||\mathcal{A}_{3}\cap\mathcal{A}_1|+|\mathcal{A}_{3}||\mathcal{A}_{1}\cap\mathcal{A}_2|).
\end{split}
\end{equation*}
Let 
$$a=|\mathcal{A}_{1}|, \,b=|\mathcal{A}_{2}|, \,c=|\mathcal{A}_{3}|, \,d=|\mathcal{A}_{2}\cap\mathcal{A}_3|, \,e=|\mathcal{A}_{3}\cap\mathcal{A}_1|, \,f=|\mathcal{A}_{1}\cap\mathcal{A}_2|.$$ If follows from (\ref{ineq2}) that
$$ab+bc+ca+ad+be+cf\le (1+o(1))4^n.$$
We also have by inclusion/exclusion
$$2^n \ge |\mathcal{A}_{1} \cup \mathcal{A}_{2} \cup \mathcal{A}_{3}|
=\sum_i |\mathcal{A}_{i}| - \sum_{i, j} |\mathcal{A}_{i} \cap \mathcal{A}_{j}|$$
which gives $a+b+c \le d+e+f+2^n$.
Thus, we are left to solve the following optimization problem. 
\begin{equation*}
\begin{aligned}
& \max & & abc \\
& \textrm{ s.t.}& & ab+bc+ca+ad+be+cf\leq (1+o(1))4^n\\
& & & a+b+c-d-e-f\leq 2^n\\
& & & d+e\leq c, e+f\leq a, f+d\leq b\\
& & & a,b,c,d,e,f\ge0.
\end{aligned}
\end{equation*}
\medskip

Now, if we rescale the variables in this optimization problem by a factor of $2^n$, that is, write $x'=x/2^n$ for $x\in\{a,b,c,d,e,f\}$, we get
\begin{equation*}
\begin{aligned}
& \max & & a'b'c' \\
& \textrm{ s.t.}& & a'b'+b'c'+c'a'+a'd'+b'e'+c'f'\le 1+\epsilon\\
& & & a'+b'+c'-d'-e'-f'\leq 1\\
& & & d'+e'\le c', e'+f'\le a', f'+d'\le b'\\
& & & a',b',c',d',e',f'\ge0.
\end{aligned}
\end{equation*}
Since $n$ can be made arbitrarily large,  $\epsilon$ is arbitrarily close to zero. So we are left to show the optimization problem with $\epsilon=0$. By solving the KKT conditions (see Appendix for the details), we get
$$\max abc\le (0.13075+o(1))2^{3n}.$$
This concludes the proof. \qed

\section{Concluding remarks}

$\bullet$ Our basic approach is simply to average over a suitable family of partitions. It can be applied to a variety of other extremal problems, for example, it yields some results about cross intersecting families proved by Borg~\cite{B14}. It also applies to the situation when the number of colors is more than the size of the forbidden configuration.  In particular, the proof of Lemma~\ref{thm1} yields the following more general statement.

\begin{lemma} \label{tk}
Given integers $s\ge 1$, $1 \le t \le k$ and  $0\le c\le s-1$, let $n$ be an integer such that $n\ge c+t(s-c)$. For $ i=1,\ldots,k$, let $\mathcal{A}_i\subset \binom{[n]}{s}$ such that $\{\mathcal{A}_i\}_{i=1}^k$ contains no multicolor sunflower with $t$ petals and core size $c$. Then,
$$\sum_{i=1}^k|\mathcal{A}_i|\le\begin{cases} \frac{(t-1)k}{m}\binom{n}{s}, & \mbox{if } c+t(s-c)\le n\le c+k(s-c) \\(t-1)\binom{n}{s}, & \mbox{if }  n\ge c+k(s-c), \end{cases}$$
where $m= \lfloor (n-c)/(s-c)\rfloor$.
\end{lemma}

Note that both upper bounds can be sharp. For the first bound, when 
 $c=0$, $m=t<k$ and $n=ms$, let each $\mathcal{A}_i$ consist of all $s$-sets omitting the element 1. A sunflower with $t=m$ petals and core size $c=0$ is a perfect matching of $[n]$. Since every perfect matching has a set containing 1, there is no multicolor sunflower.  
  Clearly  $\sum_i|\mathcal{A}_i|=k{n-1 \choose s}=((t-1)k/m){n \choose s}$.
For the second bound, we can just  take $t-1$ copies of ${[n] \choose s}$ to achieve equality.

$\bullet$
Another general approach that applies to the sum of the sizes of families was initiated by Keevash-Saks-Sudakov-Verstra\"ete~\cite{KSSV}.  Both methods can be used to solve certain problems.
For example, as pointed out to us by Benny Sudakov, the approach in~\cite{KSSV} can  be used to prove the $k=3$ case of 
Theorem~\ref{thm2}.  On the other hand, we can use our approach to prove the following  that is a very special case of a result of~\cite{KSSV}: if we have graphs $G_1, G_2, G_3$ on vertex set $[n]$ with no multicolored triangle, then
$|G_1|+|G_2|+|G_3| \le 2{n \choose 2}$ provided $n\equiv 1,3$ (mod 6).
Indeed, we just take a Steiner triple system $S$ on $[n]$ and observe (by K\"onig's theorem) that on each triple $e$ of $S$, the sum over $i$ of the number of edges of $G_i$ within $e$ is at most six (the result of~\cite{KSSV} is quite a bit stronger, as it does not require the divisibility requirement and also applies when the number of colors is much larger). The same argument works for larger cliques and even for $r$-uniform hypergraphs, using the recent result of Keevash~\cite{K} on the existence of designs. More precisely, given integers $2\le r<q$ and $n>n_0$ satisfying certain divisibility conditions, if we have $r$-uniform hypergraphs $H_1, \ldots, H_{{q \choose r}}$ on vertex set $[n]$ forming no multicolored $K_q^r$, then $\sum |H_i| \le ({q \choose r}-1){n \choose r}$ by the same proof as for triangles above except we replace a Steiner triple system with an appropriate design which is known to exist by Keevash's result.

$\bullet$
The main idea in the proof of Theorem~\ref{thm3} was to consider the graph parameter $h(G)=m_2(G)+t(G)$. The choice of $h(G)$ was obtained by a search of several parameters for which one could prove good upper bounds like in (\ref{mf}), while also being able to compute the expectation  in (\ref{ineq2}). Perhaps some new ideas will be needed to improve the upper bound further to the conjectured value of 1/8. Carrying out this approach for larger $k$ appears to be difficult.


\appendix
\section{The Optimization Problem in the Proof of Theorem~\ref{thm3}}
We are actually considering the following problem
\begin{equation*}
\begin{aligned}
& \max & & abc \\
& \textrm{ s.t.}& & ab+bc+ca+ad+be+cf-1\le 0\\
& & & a+b+c-d-e-f-1\leq 0\\
& & & a,b,c,d,e,f\geq0.
\end{aligned}
\end{equation*}
We will see that the optimal solution to this problem satisfies all the constraints in the original problem, hence solves the original problem.

We consider the KKT conditions (see, for example~\cite{BV04}) for this problem. Let $(a,b,c,d,e,f)$ be an optimal solution to the problem, and $\mu_i\ge 0, i=1,\ldots,8$ be the dual variable associated with each constraint above respectively. We have the following stationary (gradient) condition.
\begin{equation*}
\begin{split}
\begin{pmatrix}
  bc  \\ ca\\ ab\\0\\0\\0
 \end{pmatrix}
 =&\mu_1\begin{pmatrix}
 b+c+d\\a+c+e\\b+a+f\\a\\b\\c
 \end{pmatrix}
 +\mu_2\begin{pmatrix}
  1\\1\\1\\-1\\-1\\-1
 \end{pmatrix}
  +\mu_3\begin{pmatrix}
 -1\\0\\0\\0\\0\\0
 \end{pmatrix}
  +\mu_4\begin{pmatrix}
 0\\-1\\0\\0\\0\\0
 \end{pmatrix}\\
  &+\mu_5\begin{pmatrix}
 0\\0\\-1\\0\\0\\0
 \end{pmatrix}
  +\mu_6\begin{pmatrix}
 0\\0\\0\\-1\\0\\0
 \end{pmatrix}
 +\mu_7\begin{pmatrix}
 0\\0\\0\\0\\-1\\0
 \end{pmatrix}+\mu_8\begin{pmatrix}
 0\\0\\0\\0\\0\\-1
 \end{pmatrix}.
\end{split}
\end{equation*}
The complimentary slackness conditions are the following.
\begin{equation*}
\begin{split}
\mu_1(ab+bc+ca+ad+be+cf-1)&=0\\
\mu_2(a+b+c-d-e-f-1)&=0\\
\mu_3a=\mu_4b=\mu_5c=\mu_6d=\mu_7e=\mu_8f&=0.
\end{split}
\end{equation*}
First notice that clearly the maximum of $abc$ should be positive, so in any optimal solution we have $a,b,c>0$. By complimentary slackness conditions $\mu_3 a=\mu_4 b=\mu_5 c=0$, we get $\mu_3=\mu_4=\mu_5=0$.

Now, suppose that $d=e=f=0$, then the problem is reduced to maximizing $abc$, subject to $ab+bc+ca\le 1, a+b+c\le 1,$ and $a,b,c>0$. It is easy to see that in this case $a=b=c=1/3$ solves the problem, with the maximum $1/27$. Therefore, we may assume without loss of generality that $d>0$, so by complimentary slackness, $\mu_6=0.$ Thus, we have reduced the stationary condition into the following form.

$$\begin{pmatrix}
  bc  \\ ca\\ ab\\0\\0\\0
 \end{pmatrix}
 =\mu_1\begin{pmatrix}
 b+c+d\\a+c+e\\b+a+f\\a\\b\\c
 \end{pmatrix}
 +\mu_2\begin{pmatrix}
  1\\1\\1\\-1\\-1\\-1
 \end{pmatrix}
 +\begin{pmatrix}
 0\\0\\0\\0\\-\mu_7\\-\mu_8
 \end{pmatrix}.
 $$
If $\mu_1=0$, then we get that $\mu_2=0$ also, then $a=b=c=0$, a contradiction. This shows that $\mu_1>0$. Now suppose that $\mu_2=0$.  Then we get $a=\mu_2/\mu_1=0$ which makes the maximum also 0, so we must have $\mu_2>0$. By complimentary slackness, the first two constraints both hold with equality. Next, we use the stationary condition to express $a,b,c$ in terms of $\mu_1,\mu_2,\mu_7$ and $\mu_8$.
\begin{equation}
a=\frac{\mu_2}{\mu_1},\qquad b=\frac{\mu_2+\mu_7}{\mu_1},\qquad c=\frac{\mu_2+\mu_8}{\mu_1}. \label{abcd}
\end{equation}
\paragraph{Case 1. $\mu_7=\mu_8=0$.} We get $a=b=c=\mu_2/\mu_1$. In this case, we are solving the maximization of $a^3$, subject to $3a^2+ax=1, 3a-x=1,$ and $a,x>0$, where $x=d+e+f$. The optimality is obtained at $a=x=1/2$, with maximum $1/8$.

\paragraph{Case 2. Exactly one of $\mu_7$ and $\mu_8$ is positive.} Without loss of generality, we may assume $\mu_7>0$, which implies that $e=0$. Since in this case $\mu_8=0$, we get $a=c$ from~(\ref{abcd}), so the problem is reduced to
\begin{equation*}
\begin{aligned}
& \max & & a^2b \\
& \textrm{ s.t.}& & 2ab+a^2+ax-1= 0\\
& & & 2a+b-x-1= 0\\
& & & a,b,x>0
\end{aligned}
\end{equation*}
where $x=d+f$. This problem has the maximum $\frac{1}{729} \left(29+20 \sqrt{10}\right)\approx0.126537$ at $$a=\frac{1}{9} \left(1+\sqrt{10}\right)\approx0.462475,\qquad b= -\frac{-29-20 \sqrt{10}}{9 \left(1+\sqrt{10}\right)^2}\approx0.591617,$$ and $$x= -\frac{81 \left(-\frac{1}{729} 2 \left(1+\sqrt{10}\right)^3+\frac{1}{81} \left(1+\sqrt{10}\right)^2+\frac{1}{729} \left(-29-20 \sqrt{10}\right)\right)}{\left(1+\sqrt{10}\right)^2}\approx0.516568.$$

\paragraph{Case 3. Both $\mu_7,\mu_8>0$.}  This implies that $e=f=0$. So from the stationary condition, we get
\begin{equation*}
\begin{aligned}
&ac=\mu_1(a+c)+\mu_2\\
&ab=\mu_1(a+b)+\mu_2.\\
\end{aligned}
\end{equation*}
By~(\ref{abcd}), we can eliminate the variables $a,b,c$ to get
\begin{equation*}
\begin{aligned}
&\mu_2(\mu_2+\mu_8)=\mu_1^2(\mu_2+(\mu_2+\mu_8))+\mu_1^2\mu_2\\
&\mu_2(\mu_2+\mu_7)=\mu_1^2(\mu_2+(\mu_2+\mu_7))+\mu_1^2\mu_2.\\
\end{aligned}
\end{equation*}
A bit of algebra shows that $\mu_7=\mu_8$, which implies $b=c$. Thus the problem is reduced to
\begin{equation*}
\begin{aligned}
& \max & & ab^2 \\
& \textrm{ s.t.}& & 2ab+b^2+ad-1= 0\\
& & & a+2b-d-1= 0\\
& & & a,b,d>0
\end{aligned}
\end{equation*}

Similar to the previous cases, one can solve system using the method of Lagrange multipliers as follows. We first eliminate variable $d$ to get $2ab+b^2+a(a+2b-1)-1=0$. Then define the Lagrangian as
$$L(a,b,\lambda)=ab^2+\lambda(a^2+4ab+b^2-a-1).$$
To find the maximum, we need to solve the following system of equations.
\begin{equation*}
\begin{aligned}
\frac{\partial L}{\partial a}&=b^2+2\lambda a+4\lambda b-\lambda=0\\
\frac{\partial L}{\partial b}&=2ab+4\lambda a+2\lambda b=0\\
\frac{\partial L}{\partial \lambda}&=a^2+4ab+b^2-a-1=0.
\end{aligned}
\end{equation*}
One can obtain exact closed form solutions  to this system involving radicals, however, the formulas are  too long to display, so we give the approximation as follows. The maximum is approximately equal to $0.130748$, with $a\approx0.37478, b\approx0.590649, \lambda\approx-0.165171$. This corresponds to the following solution to the original problem:
$$a \approx0.37478,\qquad b=c \approx0.590649,\qquad d\approx0.556078,\qquad e=f=0.$$
Comparing with all the other cases, this is the actual maximum that $abc$ can achieve. Moreover, this solution satisfies the constraints $d+e\le c, e+f\le a$ and $f+d\le b$, hence is the optimal solution to the original problem in the proof of Theorem~\ref{thm3} as well.
\end{document}